\documentclass[10pt, leqno]{amsart}
\usepackage{amscd}
\usepackage{amssymb}
\usepackage{amsmath}
\usepackage{amsfonts}
\usepackage{stmaryrd}
\usepackage{amsthm}
\usepackage[cmtip,all]{xypic}
\usepackage{pifont}
\usepackage{faktor}

\usepackage[usenames, dvipsnames]{color}

\theoremstyle{plain}
\newtheorem{Lem}{Lemma}[section]

\newtheorem{Cor}[Lem]{Corollary}
\newtheorem{Thm}[Lem]{Theorem}

{\theoremstyle{definition} 

\newtheorem{Rk}[Lem]{Remark}
\newtheorem{Def}[Lem]{Definition}}

\newcommand{\zig}{\addtocounter{Lem}{1}\tag{\theLem}}
\newcommand{\G}{\mathbb{Z}_p^\times}
\newcommand{\cyclic}{C_{p-1}}
\newcommand{\padics}{\mathbb{Z}_p}
\newcommand{\zee}{\mathbb{Z}}

\def\:{\colon}
\DeclareMathOperator*{\colim}{colim}
\DeclareMathOperator*{\holim}{holim}

\begin{document}
\title{A computational reduction for many base cases in profinite telescopic algebraic $K$-theory}
\author{Daniel G. Davis}
\begin{abstract} 
For primes $p\geq 5 $, $K(KU_p)$---the algebraic $K$-theory spectrum of $(KU)^{\wedge}_p$, Morava $K$-theory $K(1)$, and Smith-Toda complex 
$V(1)$, Ausoni and Rognes conjectured (alongside related conjectures) that 
$L_{K(1)}S^0 \mspace{-1.5mu}\xrightarrow{\mspace{-2mu}\text{unit} \, i}~\mspace{-7mu}(KU)^{\wedge}_p$ induces a map 
$K(L_{K(1)}S^0) \wedge v_2^{-1}V(1) \to K(KU_p)^{h\G} \wedge v_2^{-1}V(1)$ that is an equivalence. Since the definition of this map is not well understood, we consider $K(L_{K(1)}S^0) \wedge v_2^{-1}V(1) \to (K(KU_p) \wedge v_2^{-1}V(1))^{h\G}$, which is induced by $i$ and also should be an equivalence. We show that for any closed $G < \G$, $\pi_\ast((K(KU_p) \wedge v_2^{-1}V(1))^{hG})$ is a direct sum of 
two pieces given by (co)invariants and a coinduced module, for 
$K(KU_p)_\ast(V(1))[v_2^{-1}]$. When $G = \G$, the direct sum is, conjecturally, 
$K(L_{K(1)}S^0)_\ast(V(1))[v_2^{-1}]$ and, by using $K(L_p)_\ast(V(1))[v_2^{-1}]$, 
where $L_p = ((KU)^{\wedge}_p)^{h\zee/((p-1)\zee)}$, 
the summands simplify. The Ausoni-Rognes conjecture suggests that 
in 
\[(-)^{h\G} \wedge v_2^{-1}V(1) \simeq (K(KU_p) \wedge v_2^{-1}V(1))^{h\G},\] 
$K(KU_p)$ fills in the blank; we show that for 
any $G$, the blank can be filled by $(K(KU_p))^\mathrm{dis}_\mathcal{O}$, a 
discrete $\G$-spectrum built out of $K(KU_p)$.
\end{abstract}

\maketitle
\section{Introduction} 
\subsection{A brief overview of the circle of ideas in which this work occurs}
The algebraic $K$-theory spectrum $K(S^0)$ of the sphere spectrum $S^0$ plays an important role in geometric topology. For example, there is an equivalence \[K(S^0) \simeq S^0 \vee 
\mathrm{Wh}^\mathrm{Diff}(\ast),\] where $\mathrm{Wh}^\mathrm{Diff}(\ast)$ is the smooth 
Whitehead spectrum of a point, so that $K(S^0)$ is closely related, for large $n$, to 
$\mathrm{Diff}(D^n)$, the group of self-diffeomorphisms 
of the $n$-disc $D^n$ that fix its boundary. Also, if we let $\mathcal{W}$ denote the infinite loop space of $\mathrm{Wh}^\mathrm{Diff}(\ast)$, then the loop space of $\mathcal{W}$ is homotopy equivalent to 
the stable $h$-cobordism space of $\ast$. For these results and other examples of the role of 
$K(S^0)$ in geometric topology, see \cite[Sections 2, 3]{rognesicm}, \cite{WJR}, and 
\cite[Section 0]{2primary}. 

In \cite{Waldhausen}, Waldhausen initiated an effort to understand $K(S^0)$ by using chromatic homotopy theory, in which one tries to understand $S^0$, by understanding 
the $K(n)$-local sphere $L_{K(n)}(S^0)$ (and other ``height $n$" objects too) 
for every natural number $n$ and each prime $p$ (see, for example, \cite{HoveyCech}, 
\cite{GHMR}, \cite[Section 1.2]{mathbbC}, and \cite[page 1, Proposition 3.2, Example 3.3]{fracturecubes}). 
Here, $K(n)$ is the $n$th Morava $K$-theory spectrum and it varies with $p$. 
As explained in \cite[pages 3--7]{acta} (also, see \cite[Section 1]{GabeTHH} for an account with 
a focus on the role of $n = 1$), it follows that studying the algebraic $K$-theory spectra $K(L_{K(n)}(S^0))$ is fruitful for 
the study of $K(S^0)$. Though this method of study is promising, it is also very 
challenging, and one of the 
main strategies for making progress with it is a conjecture by Ausoni and Rognes (see \cite[paragraph containing (0.1)]{acta}) that involves 
the extended Morava stabilizer group $G_n$. To avoid too much detail in this opening subsection, we omit the statement of the conjecture here; we discuss in detail 
what we need from 
it below. 

The base case of this conjecture, when $n = 1$, is currently unsolved. For this base case, when $p \geq 5$, we 
make progress in this paper on a computation that gives a conjectural description of certain 
stable homotopy 
groups that appear in the conjecture (for a particular choice of ``telescope" (see below)) and are closely related to the stable 
homotopy groups of $K(L_{K(1)}(S^0))$. 
\subsection{A closer look at the motivation for our work}  

Let $p$ be any prime, with $\mathbb{Z}_p$ the $p$-adic integers, and let 
$L_{K(1)}(S^0)$ be the Bousfield localization of the sphere spectrum with respect to 
the first Morava $K$-theory spectrum. Also, 
let $KU_p$ be $p$-complete complex $K$-theory, so that 
\[\pi_\ast(KU_p) = \mathbb{Z}_p[u^{\pm 1}],\] where $\pi_0(KU_p) = \mathbb{Z}_p$ 
and $|u| = 2$, and let $\mathbb{Z}_p^\times$ denote the group of units in 
$\mathbb{Z}_p$. By \cite{Pgg/Hop0, AndreQuillen}, $\mathbb{Z}_p^\times$ -- as the group of $p$-adic Adams operations -- acts on the commutative $S^0$-algebra 
$KU_p$ by maps of 
commutative $S^0$-algebras. Given a commutative $S^0$-algebra $A$, 
the algebraic $K$-theory spectrum of $A$, $K(A)$, 
is a commutative $S^0$-algebra, so that $K(KU_p)$ is a commutative 
$S^0$-algebra, and by 
the functoriality of $K(-)$, $\mathbb{Z}_p^\times$ acts on 
$K(KU_p)$ by maps of commutative $S^0$-algebras.

For the rest of this paper, we let $p \geq 5$. Let 
$V(1)$ be the type $2$ Smith-Toda 
complex $S^0/(p, v_1)$. Then there is a $v_2$-self-map 
$v \: \Sigma^d V(1) \to V(1)$, where $d$ is 
some positive integer (see \cite[Theorem 9]{nilpotencetwo}), and hence, 
$v$ induces a sequence 
\[V(1) \to \Sigma^{-d}V(1) \to \Sigma^{-2d} V(1) \to \cdots\] of maps of 
spectra, and we set
\[v_{2}^{-1}V(1) = \colim_{j \geq 0} \Sigma^{-jd}V(1),\] the 
mapping telescope associated to $v$. 
In \cite[paragraph containing (0.1)]{acta}, \cite[Conjecture 4.2]{rognesguido}, 
and \cite[page 46; Remark 10.8]{jems},
Christian Ausoni and John Rognes conjectured that the $K(1)$-local unit map 
\[i \: L_{K(1)}(S^0) \to KU_p\] induces a 
weak equivalence 
\begin{equation}\label{themap}\zig
K(L_{K(1)}(S^0)) \wedge v_2^{-1}V(1) \to K(KU_p)^{h\mathbb{Z}_p^\times} 
\wedge v_2^{-1}V(1),\end{equation} where \[K(KU_p)^{h\mathbb{Z}_p^\times} = (K(KU_p))^{h\mathbb{Z}_p^\times}\] is a 
continuous homotopy fixed point spectrum that is formed with respect to a continuous action of the profinite group $\mathbb{Z}_p^\times$ on 
$K(KU_p)$. 

\begin{Rk}
The above conjecture is a 
collection of $n=1$ instances of a more general conjecture made by Ausoni and Rognes for every positive integer and every prime (for 
more information, see the references mentioned above). 
\end{Rk}

One difficulty with making progress on this conjecture is that there is no published construction of $K(KU_p)^{h\mathbb{Z}_p^\times}$ 
and, according to \cite[Remark 1.5]{padicspectra2}, the only models for it, 
currently, are a ``candidate definition" that uses condensed 
spectra (in the sense of Clausen-Scholze) in the setting of $\infty$-categories 
(the author learned of this construction from Jacob Lurie) and, possibly, a 
pyknotic version of this construction (in the framework of \cite{pyknotic}). 
Thus, due to the lack of a robust model for the map in (\ref{themap}), the 
conjecture is difficult to approach computationally. 

If $G$ is any profinite group and $X$ is a discrete $G$-spectrum (as in 
\cite{joint}; the crux of this concept is that for every $k$, $l \geq 0$, the 
set of $l$-simplices of the pointed simplicial set $X_k$ is a discrete $G$-set), then 
there is a continuous homotopy fixed point spectrum $X^{hG}$ 
\cite[Section 3.1]{joint} (and we use this notation for the rest of this paper). 
Thus, to address the above difficulty, the author showed in \cite[Section 1.2]{padicspectra2} that $K(KU_p)
\wedge v_2^{-1}V(1)$, with $v_2^{-1}V(1)$ equipped with the trivial $\G$-action, 
can be realized as a discrete $\G$-spectrum -- written as 
$C^\mathrm{dis}_p$ in \cite{padicspectra2}, and hence, one can form 
\[(K(KU_p)
\wedge v_2^{-1}V(1))^{h\mathbb{Z}_p^\times} := (C^\mathrm{dis}_p)^{h\G}\] 
and, by \cite[Theorem 1.8]{padicspectra2}, 
the map $i$ induces a canonical map 
\[i' \: K(L_{K(1)}(S^0)) \wedge v_2^{-1}V(1) \to (K(KU_p)
\wedge v_2^{-1}V(1))^{h\mathbb{Z}_p^\times}.\] 

\begin{Rk}
According to \cite[Remark 1.5]{padicspectra2}, the relationship between 
the target of $i'$ and $\mathbb{K} \wedge v_2^{-1}V(1)$, where $\mathbb{K}$ 
denotes the aforementioned candidate model for 
$K(KU_p)^{h\mathbb{Z}_p^\times}$, is unclear.
\end{Rk}

Now we make some observations to understand the relationship between the map 
$i'$ and the conjectural equivalence in (\ref{themap}). If $X$ is a discrete $\G$-spectrum and $Y$ is a 
finite spectrum with trivial $\G$-action, 
then $X \wedge Y$ is a discrete $\G$-spectrum and, by \cite[Remark 7.16]{cts},
\begin{equation}\label{instanceA}\zig
(X \wedge Y)^{h\G} \simeq X^{h\G} \wedge Y.\end{equation}
More generally, if $\{X_i\}_{i \in I}$ is a diagram of discrete $\G$-spectra indexed by a cofiltered category $I$, then the equivalence
\[(\holim_i X_i) \wedge Y \simeq \holim_i (X_i \wedge Y)\] implies that it is 
natural to make the definition 
\[((\holim_i X_i) \wedge Y)^{h\G} := (\holim_i (X_i \wedge Y))^{h\G} = \holim_i (X_i \wedge Y)^{h\G},\] where the last step applies \cite[Section 4.4]{joint}, and thus, 
we have
\begin{equation}\label{instanceB}\zig
((\holim_i X_i) \wedge Y)^{h\G} \simeq (\holim_i X_i)^{h\G} \wedge Y,\end{equation}
because
\begin{align*}
\holim_i (X_i \wedge Y)^{h\G} \simeq 
\holim_i &((X_i)^{h\G} \wedge Y) \\ &\simeq (\holim_i (X_i)^{h\G}) \wedge Y
 = (\holim_i X_i)^{h\G} \wedge Y.\end{align*} 
Also, by \cite{padicspectra2}, for each $j \geq 0$, $K(KU_p) \wedge \Sigma^{-jd}V(1)$ 
can be realized as a discrete $\G$-spectrum, and hence, there is 
$(K(KU_p) \wedge \Sigma^{-jd}V(1))^{h\G}.$
Then since each $\Sigma^{-jd}V(1)$ is a finite spectrum with trivial $\G$-action, 
the pattern in (\ref{instanceA}) and (\ref{instanceB}) suggests that 
there should be an equivalence 
\begin{equation}\label{hope}\zig
K(KU_p)^{h\mathbb{Z}_p^\times} \wedge \Sigma^{-jd}V(1) 
\overset{?}{\simeq} (K(KU_p) \wedge \Sigma^{-jd}V(1))^{h\G}.\end{equation} 
Here and elsewhere, we place a ``?" over a relation to indicate that it is not 
known to be true, but it is desired and expected to some degree.

Now notice that there is the isomorphism 
\begin{equation}\label{iso}\zig
K(KU_p)^{h\mathbb{Z}_p^\times} \wedge v_2^{-1}V(1) 
\cong
 \colim_{j \geq 0}(K(KU_p)^{h\mathbb{Z}_p^\times} \wedge \Sigma^{-jd}V(1))\end{equation} 
 and, by \cite[Theorem 1.7]{padicspectra2}, there is 
an equivalence 
\begin{equation}\label{equiv1}\zig
(K(KU_p) \wedge v_2^{-1}V(1))^{ h\mathbb{Z}_p^\times} 
\simeq \colim_{j \geq 0} \mspace{1mu}(K(KU_p) \wedge \Sigma^{-jd}V(1))^{h\mathbb{Z}_p^\times}.\end{equation} 
Thus, (\ref{hope})--(\ref{equiv1}) imply that there should be an 
equivalence 
\begin{equation}\label{equiv2}\zig
K(KU_p)^{h\mathbb{Z}_p^\times} \wedge v_2^{-1}V(1) 
\overset{?}{\simeq}
(K(KU_p) \wedge v_2^{-1}V(1))^{ h\mathbb{Z}_p^\times},\end{equation} and this observation suggests that if 
(\ref{equiv2}) holds and $i'$ is a weak equivalence, then one should be able to prove that the map in (\ref{themap}) is a weak equivalence, and thereby verify the conjecture of Ausoni and Rognes. This potentially fruitful strategy for 
proving this conjecture involves computing 
\[\pi_\ast\bigl((K(KU_p) \wedge v_2^{-1}V(1))^{h\mathbb{Z}_p^\times}\bigr),\] 
and thus, in this paper, we make progress on this computation by showing that it is 
a 
direct sum of two pieces given by invariants and coinvariants involving the 
$\G$-action on $\pi_\ast(K(KU_p) \wedge v_2^{-1}V(1))$. Additionally, with 
\[L_p := (KU_p)^{h\mathbb{Z}/((p-1)\zee)}\] (as in \cite{acta}), the $p$-complete Adams summand and a commutative $S^0$-algebra, where $\mathbb{Z}/((p-1)\zee)$ is the usual subgroup of $\G$, we show that the direct sum can be expressed as 
invariants and coinvariants of the $\padics$-action on 
$\pi_\ast(K(L_p) \wedge v_2^{-1}V(1))$.

Given a profinite group $G$ and a discrete $G$-spectrum $X$, if $H$ is any closed subgroup of $G$, then 
$H$ is a profinite group, $X$ is a discrete $H$-spectrum (by restriction of the 
$G$-action), and hence, there is the continuous homotopy fixed point spectrum $X^{hH}$. Our work for the above computation is in line with this multiplicity of 
possibilities: our result is not just for the $\G$-homotopy fixed points, but is for the homotopy fixed points of any closed subgroup (though the aforementioned presentation 
involving $L_p$ is only for the case $G=H=\G$). 

\subsection{The main results}

In (\ref{equiv2}) above, we said that there should be an equivalence
\[(-)^{h\mathbb{Z}_p^\times} \wedge v_2^{-1}V(1) \overset{?}{\simeq} 
(K(KU_p) \wedge v_2^{-1}V(1))^{h\mathbb{Z}_p^\times},\] where 
the blank ``$-$" can be filled in with $K(KU_p)$. One of our intermediate steps 
in obtaining the results mentioned above is to give a way to fill in this blank with a discrete 
$\G$-spectrum that is related to $K(KU_p)$.  

\begin{Thm}\label{prehfp}
Let  $p \geq 5$ and let $G$ be a closed subgroup of $\G$. There is a discrete $\G$-spectrum $(K(KU_p))^\mathrm{dis}_\mathcal{O}$ with the property that 
for each $j \geq 0$, there is an equivalence 
\[\bigl((K(KU_p))^\mathrm{dis}_\mathcal{O}\bigr)^{\mspace{-3mu}hG} \wedge \Sigma^{-jd}V(1) 
\simeq 
(K(KU_p) \wedge \Sigma^{-jd}V(1))^{\mspace{-0mu}hG},\] and
\[\bigl((K(KU_p))^\mathrm{dis}_\mathcal{O}\bigr)^{\mspace{-3mu}hG} \wedge
v_2^{-1}V(1) \simeq 
(K(KU_p) \wedge v_2^{-1}V(1))^{hG}.\]
\end{Thm}

The spectrum $(K(KU_p))^\mathrm{dis}_\mathcal{O}$ is defined in Definition \ref{xdisO}, with $\mathcal{O}$ specified at the beginning of Section \ref{step3}, 
the first equivalence in Theorem \ref{prehfp} is Theorem \ref{hfp}, and 
the last equivalence follows immediately from the first one and the general version 
of (\ref{equiv1}) that is stated in (\ref{colim}).

After the following prefatory remarks, we state our result that for any closed 
subgroup $G$ in $\G$, $\pi_\ast\bigl((K(KU_p) \wedge v_2^{-1}V(1))^{hG}\bigr)$ 
can be reduced to a direct sum.  

Recall that $\padics$ is the pro-$p$ completion of $\zee$ and $\zee$ can 
be regarded as a subset of $\padics$ in a way that makes the inclusion 
$\zee \hookrightarrow \padics$ a ring homomorphism. We define
\[C_{p-1} : = \mathbb{Z}/((p-1)\mathbb{Z})\] to be the cyclic group 
of order $p-1$ and recall that 
\[\G \cong \padics \times \cyclic\] (since $p \geq 5$). If $M$ is a $\zee[\padics]$-module (so that $\padics$ acts on $M$), then $M$ is naturally a $\zee[\zee]$-module, and $M_{_{\mspace{-.2mu}\scriptstyle{\zee}}}$ denotes the coinvariants. If $K$ is a closed subgroup of a profinite group $H$ and $A$ is a discrete $K$-module, we let $\mathrm{Coind}^H_K(A)$ denote the coinduced discrete $H$-module of continuous $K$-equivariant functions $H \to A$. 

Let $P(v_2) = \mathbb{F}_p[v_2]$ denote the polynomial algebra over $\mathbb{F}_p$ generated by the periodic element $v_2 \in 
\pi_{2p^2-2}(V(1))$. Also, $P(v_2^{\pm 1}) = \mathbb{F}_p[v_2, v_2^{-1}]$ is the 
algebra of Laurent polynomials on $v_2$. To help manage the typography in the 
upcoming text, for a closed subgroup $G$ of $\G$, we let 
\[\mathbb{KV}(p,G) := \bigl((K(KU_p))^\mathrm{dis}_\mathcal{O}\bigr)^{\mspace{-3mu}hG} \wedge v_2^{-1}V(1)\] and 
\[\mathcal{E}(p,G)_\ast := H^1_c(G, \pi_{\ast}(K(KU_p) \wedge v_2^{-1}V(1))),\] 
a graded continuous cohomology group with coefficients in the stated discrete $G$-module. 

\begin{Thm}\label{directsum}
Let $p \geq 5$ and let $G$ be any closed subgroup of $\G$. There is an isomorphism 
\[\pi_\ast\bigl((K(KU_p) \wedge v_2^{-1}V(1))^{hG}\bigr) \cong 
\pi_{\ast}(\mathbb{KV}(p,G)),\] 
where the right-hand side is the middle term in a
short exact sequence 
\[0 \to \mathcal{E}(p,G)_{\ast+1} \to \pi_{\ast}(\mathbb{KV}(p,G)) \to \bigl(\pi_{\ast}(K(KU_p) \wedge v_2^{-1}V(1))\bigr)^{\mspace{-3mu}G} \to 0\] of $P(v_2^{\pm1})$-modules. 
In particular, in each degree $t$, where $t \in \zee$, this sequence is a split exact sequence  of $\,\mathbb{F}_p$-modules and there is an 
isomorphism
\begin{align*}
\pi_t\bigl(&(K(KU_p) \wedge v_2^{-1}V(1))^{hG}\bigr) \\ & \cong 
\bigl(\mspace{-2mu}\bigl(\mathrm{Coind}^{\G}_{G}\mspace{-4mu}(\pi_{t+1}(K(KU_p) \wedge v_2^{-1}V(1)))\bigr)^{\mspace{-2mu}C_{p-1}}\bigr)_{_{\mspace{-2.5mu}\scriptstyle{\zee}}} \oplus 
\bigl(\pi_{t}(K(KU_p) \wedge v_2^{-1}V(1))\bigr)^{\mspace{-3mu}G}\end{align*} of abelian groups, where in the direct sum, the left summand is isomorphic to 
$\mathcal{E}(p,G)_{t+1}$.
\end{Thm} 

The proof of this result is broken up into six steps: 
\begin{itemize}
\item
in Section \ref{stepi}, we use various homotopy fixed point spectral sequences to 
present $\pi_\ast\bigl((K(KU_p) \wedge v_2^{-1}V(1))^{hG}\bigr)$ as the middle term in a colimit of short exact 
sequences;
\item
Section \ref{step2} makes some recollections of several constructions that are 
needed to go further;
\item
for each $j \geq 0$, $(K(KU_p) 
\wedge \Sigma^{-jd}V(1))^{hG}$ is the continuous 
homotopy fixed points of, not literally, $K(KU_p) 
\wedge \Sigma^{-jd}V(1)$, but a discrete $\G$-spectrum equivalent to 
this $\G$-spectrum, and in Section \ref{step3}, we study the role of $V(1)$ in 
the construction of this discrete $\G$-spectrum and its associated homotopy fixed point spectral 
sequence (and thereby prove Theorem \ref{prehfp}, of which the first isomorphism in 
Theorem \ref{directsum} is an immediate consequence); 
\item 
Section \ref{step4} shows that each of the just-mentioned spectral sequences 
is isomorphic to a spectral sequence in the category of $P(v_2)$-modules; 
\item
in Section \ref{step5}, we obtain the desired short exact sequence of $P(v_2^{\pm1})$-modules and the chief desideratum is shown to be a direct sum with 
its second summand as specified in Theorem \ref{directsum}; and 
\item
we obtain the isomorphism between 
$\mathcal{E}(p,G)_{t+1}$ and the expression involving $\mathbb{Z}$-coinvariants of 
$C_{p-1}$-invariants (in every integral degree $t$) in Section \ref{step6}.
\end{itemize}

\begin{Rk}\label{ctscoinvariants}
In Lemma \ref{not_restricted}, we show that in Theorem \ref{directsum}, for each integer $t$,
\[\mathcal{E}(p,G)_{t+1} \cong
\padics {\otimes}_{\padics[[\padics]]} \bigl(\bigl(\mathrm{Coind}^{\G}_{G}\mspace{-3mu}(\pi_{t+1}(K(KU_p) \wedge v_2^{-1}V(1)))\bigr)^{\mspace{-3mu}C_{p-1}}\bigr),\] where $\padics$ is regarded as a $\padics[[\padics]]$-module by giving $\padics$ the trivial $\padics$-group action. We give this result in case this 
form of $\mathcal{E}(p,G)_{t+1}$ is easier to compute than the $\zee$-coinvariants 
of Theorem \ref{directsum}. We point out that ``$\otimes_{\padics[[\padics]]}$" above denotes the usual tensor product (for the category of abstract $\padics[[\padics]]$-modules) and not a completed tensor product (formed in some category of topological $\padics[[\padics]]$-modules). 
\end{Rk} 

Now we focus on the case $G = \G$: our result in this case -- Corollary \ref{cor} below -- consists of three 
isomorphisms, and the first one is an immediate consequence of 
Theorem \ref{directsum}. As alluded to earlier, the last two isomorphisms involve 
$K(L_p)$, and so we note that $\pi_\ast(L_p) = \padics[v_1^{\pm 1}]$ and $L_p \simeq E(1)_p$, the $p$-completed 
first Johnson-Wilson spectrum. Also, we make explicit the following, which was 
implicitly referred to earlier: after taking $C_{p-1}$-homotopy fixed points to form $L_p$, 
there is a residual action by $\padics$ on $L_p$ 
through morphisms of commutative $S^0$-algebras, and hence, $K(L_p)$ carries a $\padics$-action. The telescope $v_2^{-1}V(1)$ is given the trivial $\padics$-action, 
and $K(L_p) \wedge v_2^{-1}V(1)$ is equipped with the diagonal $\padics$-action.

\begin{Cor}\label{cor}
Let $p \geq 5$. There are isomorphisms
\begin{align*}
\pi_\ast\bigl((K & (KU_p) \wedge v_2^{-1}V(1))^{h\G}\bigr)\\ & \cong 
\bigl(\mspace{-2mu}\bigl(\pi_{\ast+1}(K(KU_p) \wedge v_2^{-1}V(1))\bigr)^{\mspace{-2mu}C_{p-1}}\bigr)_{{\mspace{-1.5mu}\scriptstyle{\zee}}} \oplus 
\bigl(\pi_\ast(K(KU_p) \wedge v_2^{-1}V(1))\bigr)^{\G} \\ 
& \cong \bigl(\pi_{\ast+1}(K(L_p) \wedge v_2^{-1}V(1))\bigr)_{{\mspace{-1.5mu}\scriptstyle{\zee}}} \oplus \bigl(K(L_p)_\ast(V(1))[v_2^{-1}]\bigr)^{\padics} \\ 
& \cong \bigl(\padics \otimes_{\padics[[\padics]]}\mspace{-4mu}\bigl(\pi_{\ast+1}(K(L_p) \wedge v_2^{-1}V(1))\bigr)\bigr) \oplus \bigl(K(L_p)_\ast(V(1))[v_2^{-1}]\bigr)^{\padics}.\end{align*}
\end{Cor} 

\begin{Rk}
As discussed in more detail in Section \ref{1.3}, the Ausoni-Rognes conjecture suggests that for $p \geq 5$, the direct sum in Corollary \ref{cor} -- expressed in three different, but isomorphic, ways -- is a conjectural description of 
\[\pi_\ast(K(L_{K(1)}(S^0)) \wedge v_2^{-1}V(1)) \cong K(L_{K(1)}(S^0))_\ast(V(1))[v_2^{-1}],\] and it seems that it would be helpful to have a more explicit form of this 
direct sum. We note that 
\cite[page 4; Theorem 1.5]{GabeTHH} describes a strategy for computing 
$\pi_\ast(K(L_{K(1)}(S^0)) \wedge V(1))$ and gives a result that begins making progress on this strategy.
\end{Rk}

The second isomorphism of Corollary \ref{cor} comes from the first one and the fact that there is a 
$\padics$-equivariant isomorphism
\[(\pi_\ast(K(KU_p) \wedge v_2^{-1}V(1)))^{C_{p-1}} 
\cong K(L_p)_\ast(V(1))[v_2^{-1}],\] which is deduced in Section \ref{lastmaybe} 
from the fact that 
$K(L_p)$ and $K(KU_p)^{hC_{p-1}}$ are equivalent after $p$-completion (for this 
equivalence, see 
\cite[the sentence above Remark 4.4]{rognesicm}; a proof is in \cite[pages 11-12]{jems}). The third isomorphism in the corollary is an application of Remark \ref{ctscoinvariants}. 

\subsection{Considerations for the future, terminology, and notation}\label{1.3}

In our discussion of (\ref{equiv2}), we saw that proving that $i'$ is a weak equivalence would be a substantial step towards verifying the Ausoni-Rognes conjecture (more precisely, the instances described earlier of this general conjecture), and by 
Corollary \ref{cor}, this step can be done by showing that $i'$ induces 
an isomorphism 
\begin{align*}
K(L_{K(1)} & (S^0))_\ast(V(1))[v_2^{-1}] \\  
& \overset{?}{\cong} 
\bigl(\mspace{-2mu}\bigl(\pi_{\ast+1}(K(KU_p) \wedge v_2^{-1}V(1))\bigr)^{\mspace{-2mu}C_{p-1}}\bigr)_{{\mspace{-1.5mu}\scriptstyle{\zee}}} \oplus 
\bigl(\pi_\ast(K(KU_p) \wedge v_2^{-1}V(1))\bigr)^{\G}.
\end{align*} 
In \cite[Theorem 8.3]{ausoniinventiones}, under the assumption of two 
hypotheses, there is a description of the graded abelian group 
$\pi_\ast(K(KU_p) \wedge V(1))$ as a 
certain type of module (see [ibid.] for the details), and progress in verifying this description was 
made by \cite[page 2; Theorem 4.5]{BlumbergMandellMemoir}. 

Also by Corollary \ref{cor}, we see that another and perhaps easier way to 
take the aforementioned step is to prove that $i'$ induces 
an isomorphism
\[K(L_{K(1)}(S^0))_\ast(V(1))[v_2^{-1}] \mspace{-2mu} \overset{?}{\cong} \mspace{-2mu}
\bigl(\pi_{\ast+1}(K(L_p) \wedge v_2^{-1}V(1))\bigr)_{{\mspace{-3mu}\scriptstyle{\zee}}} \mspace{-2mu}\oplus \mspace{-4mu}\bigl(K(L_p)_\ast(V(1))[v_2^{-1}]\bigr)^{\padics}.\] We make a comment related to 
computing more explicitly the right-hand side of this conjectural isomorphism. 
By \cite{localization} (as conjectured in \cite[page 5]{acta}), there is a localization cofiber sequence
\[K(\padics) \to K(\ell_p) \to K(L_p) \to \Sigma K(\padics),\] where 
$\ell_p$ is the $p$-complete connective Adams summand, with $\pi_\ast(\ell_p) = \padics[v_1]$. Thus, there is the cofiber sequence 
\[K(\padics) \wedge V(1) \to K(\ell_p) \wedge V(1)\to K(L_p) \wedge V(1) \to \Sigma K(\padics) \wedge V(1),\] and as stated in \cite[page 1267]{rognesicm}, from explicit computations of $K(\padics)_\ast(V(1))$ (known by 
\cite{asterisque}; see also \cite[page 664]{ausoniinventiones}) and $K(\ell_p)_\ast(V(1))$ \cite[Theorem 9.1]{acta}, the long exact sequence for this 
cofiber sequence yields calculations of $K(L_p)_\ast(V(1))$, and some information about this is in 
\cite[Example 5.3]{rognesicm}. 

The author did not push the computation of the last-mentioned ``right-hand side" further and one reason is a lack of knowledge about the $\padics$-action on $K(L_p)_\ast(V(1))$. In this vein, we note that \cite[Remark 1.4]{acta} mentions a gap in 
understanding of how a certain Adams operation on $K(\ell_p)$ acts on a particular class in $K_{2p-1}(\ell_p)$ (we refer the reader to [ibid.] for the details).

In this paper, we always work in 
the category $Sp^\Sigma$ of symmetric spectra of 
simplicial sets, so that 
``spectrum" always means symmetric spectrum (except for a few places in 
this introduction, where the context makes the meaning clear). We let 
\[(-)_f \: Sp^\Sigma \rightarrow 
Sp^\Sigma, \ \ \ Z \mapsto Z_f\] 
denote a fibrant replacement functor, so that given the spectrum $Z$, 
there is a natural map $Z \rightarrow Z_f$ that is a trivial cofibration, 
with $Z_f$ 
fibrant. If $K$ is any group and $X$ is a $K$-spectrum, then 
$X_f$ is also 
a $K$-spectrum and the trivial cofibration 
$X \rightarrow X_f$ is $K$-equivariant.  

Given a spectrum $Z$ and an integer $t$, $\pi_t(Z)$ denotes $[S^t, Z]$, the 
set of morphisms $S^t \to Z$ in the homotopy category 
of symmetric spectra, where here, $S^t$ denotes a fixed cofibrant and fibrant model 
for the $t$-th suspension of the sphere spectrum. Outside of this introduction, 
``$\holim$" denotes the homotopy limit for $Sp^\Sigma$, as defined in 
\cite[Definition 18.1.8]{hirschhorn}. If $Z^\bullet$ is a cosimplicial spectrum that is 
objectwise fibrant, then by ``the homotopy spectral sequence for $\holim_\Delta Z^\bullet$," we mean the conditionally convergent spectral sequence 
\[E_2^{s,t} = H^s[\pi_t(Z^\ast)] \Longrightarrow \pi_{t-s}(\holim_\Delta Z^\bullet),\] 
where $\pi_t(Z^\ast)$ is the usual cochain complex associated to the cosimplicial 
abelian group $\pi_t(Z^\bullet)$.

\subsection*{Acknowledgement}
I thank Birgit Richter for helpful comments that were made after 
the results and proofs in this paper were completed. 

\section{Step i: a reduction to a colimit of short exact sequences}\label{stepi}

Let $G$ be any closed subgroup of $\G$. If $M$ is a discrete $G$-module, then we let $H^\ast_c(G,M)$ denote the continuous cohomology groups of $G$ with coefficients in $M$. By \cite[Theorem 1.7]{padicspectra2}, 
there is a 
strongly convergent homotopy spectral sequence $\{E_r^{\ast,\ast}\}_{r \geq 1} 
= \{E_r^{\ast,\ast}\}$ that 
has the form
\[E_2^{s,t} = H^s_c(G, \pi_t(K(KU_p) \wedge V(1))[v_2^{-1}]) \Longrightarrow 
\pi_{t-s}\bigl((K(KU_p) \wedge v_2^{-1}V(1))^{ hG}\bigr),\]
with $E_2^{s,t} = 0$, for all $s \geq 2$, $t \in \mathbb{Z}$. Since the $E_2$-page 
has only two nontrivial columns, there is a short 
exact sequence
\begin{equation}\label{ses1}\zig 
0 \to E_2^{1,t+1} \to \pi_{t}\bigl((K(KU_p) \wedge v_2^{-1}V(1))^{ hG}\bigr) \to E_2^{0,t} \to 0,\end{equation} 
for each $t \in \mathbb{Z}$. 

By \cite[Theorem 1.7]{padicspectra2}, there is  
an equivalence of spectra 
\begin{equation}\label{colim}\zig
(K(KU_p) \wedge v_2^{-1}V(1))^{ hG} 
\simeq \colim_{j \geq 0} \mspace{1mu}(K(KU_p) \wedge \Sigma^{-jd}V(1))^{hG}.
\end{equation}
This result, coupled with the fact that $H^\ast_c(G, -)$ commutes with colimits of discrete $G$-modules indexed by directed posets, implies that for every $t \in \zee$, the three nontrivial terms in (\ref{ses1}) satisfy the following: 
\begin{align*}
E_2^{1,t+1} & \cong \colim_{j \geq 0} 
H^1_c(G,\pi_{t+1}(K(KU_p) \wedge \Sigma^{-jd}V(1))),\\
\pi_{t}\bigl((K(KU_p) \wedge v_2^{-1}V(1))^{ hG}\bigr) 
& \cong \colim_{j \geq 0} \pi_t\bigl((K(KU_p) \wedge \Sigma^{-jd}V(1))^{hG}\bigr),\\
E_2^{0,t} & \cong \colim_{j \geq 0} 
\bigl(\pi_t(K(KU_p) \wedge \Sigma^{-jd}V(1))\bigr)^{\mspace{-3mu}G}.
\end{align*} Also, 
for each $j \geq 0$, by \cite[Remark 1.20, Theorem 7.6, (8.3)]{padicspectra2}, there is a strongly convergent homotopy spectral sequence $\{{^j\mspace{-3mu}E_r^{\ast,\ast}}\}$ having the form
\[{^j\mspace{-3mu}E_2^{s,t}} = H^s_c(G, \pi_t(K(KU_p) \wedge \Sigma^{-jd}V(1)))
\Longrightarrow 
\pi_{t-s}\bigl((K(KU_p) \wedge \Sigma^{-jd}V(1))^{ hG}\bigr),\]
with ${^j\mspace{-3mu}E_2^{s,t}} = 0$, for all $s \geq 2$, $t \in \mathbb{Z}$, so that there is a 
short exact sequence 
\begin{equation}\label{ses2}\zig
0 \to {^j\mspace{-3mu}E_2^{1,t+1}} \to \pi_{t}\bigl((K(KU_p) \wedge \Sigma^{-jd}V(1))^{ hG}\bigr) \to {^j\mspace{-3mu}E_2^{0,t}} \to 0,
\end{equation}
where $t \in \zee$. 

The above facts allow us to conclude that 
spectral sequence $\{E_r^{\ast,\ast}\}$ is the colimit over $\{j \geq 0\}$ 
of the spectral sequences $\{{^j\mspace{-3mu}E_r^{\ast,\ast}}\}$, and 
hence, the 
short exact 
sequence in (\ref{ses1}) is the colimit over $\{j \geq 0\}$ of the 
short exact sequences in (\ref{ses2}). More explicitly, 
there is a commutative diagram
\[\xymatrix@C-=0.35cm{
0 \ar[r] & E_2^{1,\ast+1} \ar[r] & \pi_{\ast}\bigl((K(KU_p) \wedge v_2^{-1}V(1))^{ hG}\bigr) \ar[r] & E_2^{0,\ast} \ar[r] & 0\\ 
0 \ar[r] & \displaystyle{\colim_{j \geq 0} {^j\mspace{-3mu}E_2^{1,\ast+1}}} \ar[r] \ar[u]^-\cong & \displaystyle{\colim_{j \geq 0} \pi_{\ast}\bigl((K(KU_p) \wedge \Sigma^{-jd}V(1))^{ hG}\bigr)} \ar[r]  \ar[u]^-\cong & \displaystyle{\colim_{j \geq 0} {^j\mspace{-3mu}E_2^{0,\ast}}} \ar[r]  \ar[u]^-\cong & 0
}\] in which the rows are exact and the columns are isomorphisms. 

\section{Step ii: a recollection of various constructions with spectra}\label{step2}

To go further, we need to better understand 
spectral sequence $\{{^j\mspace{-3mu}E_r^{\ast,\ast}}\}$, for each $j \geq 0$, and 
to do this, we need to recall 
several constructions. In this section, $H$ is an arbitrary profinite group. 

Given a spectrum $Z$, let $\mathrm{Sets}(H,Z)$ be the $H$-spectrum whose $k$th pointed simplicial set 
$\mathrm{Sets}(H,Z)_k$ has $l$-simplices $\mathrm{Sets}(H,Z)_{k,l}$ 
equal to the $H$-set $\mathrm{Sets}(H,Z_{k,l})$ of all functions 
$H \to Z_{k,l}$, for each $k, l \geq 0$, 
where the $H$-action on $\mathrm{Sets}(H,Z_{k,l})$ is defined by 
\begin{equation}\label{action}\zig
(h \cdot f)(h') = f(h'h), \ \ \ f \in \mathrm{Sets}(H,Z_{k,l}), \ h, h' \in H.
\end{equation} As 
explained in \cite[Section 2]{padicspectra2}, given any $H$-spectrum $X$, there is a 
cosimplicial $H$-spectrum $\mathrm{Sets}(H^{\bullet + 1}, X)$, where for 
each $n \geq 0$, the spectrum of $n$-cosimplices of $\mathrm{Sets}(H^{\bullet + 1}, X)$ is 
obtained by applying $\mathrm{Sets}(H,-)$ iteratively $n+1$ times to $X$. 

\begin{Def}\label{xdisO}
Let $X$ be an $H$-spectrum and let $\mathcal{O} = \{N_\lambda\}_{\lambda \in \Lambda}$ be an inverse 
system of open normal subgroups of $H$ ordered by inclusion, 
over a directed poset $\Lambda$. Following \cite[Definition 4.4]{padicspectra2}, 
\begin{align*} 
X^\mathrm{dis}_\mathcal{O} := \colim_{\lambda \in \Lambda} 
\holim_\Delta \mathrm{Sets}(H^{\bullet + 1}, X_f)^{N_\lambda},\end{align*}
where the colimit is formed 
in spectra (this definition is slightly more general than that of ``$\,X^\mathrm{dis}_\mathcal{N}\,$" in [ibid.]: $\mathcal{N}$ satisfies several hypotheses that we do not 
require from $\mathcal{O}$). Each spectrum $\holim_\Delta \mathrm{Sets}(H^{\bullet + 1}, X_f)^{N_\lambda}$ is an $H/N_\lambda$-spectrum, and hence, a discrete $H$-spectrum, via the canonical 
projection $H \to H/N_\lambda$, so that 
$X^\mathrm{dis}_\mathcal{O}$ is a discrete $H$-spectrum. Also, $X^\mathrm{dis}_\mathcal{O}$ is a fibrant spectrum (this follows from \cite[steps taken between 
(4.12) and (4.13)]{padicspectra2} and the fact that a homotopy limit of fibrant spectra is again fibrant).
\end{Def}

Let $\mathcal{O}$ be as in Definition \ref{xdisO}. 
By \cite[Lemma 4.7, proof of Theorem 4.9]{padicspectra2}, 
for any $H$-spectrum $X$, there is a zigzag 
\[X \underset{i_X}{\xrightarrow{\,\simeq\,}} \holim_\Delta \mathrm{Sets}(H^{\bullet+1}, X_f) \xleftarrow{\phi_X} X^\mathrm{dis}_\mathcal{O}\] of $H$-equivariant maps, where $i_X$ is a weak equivalence of spectra and $\phi_X$ 
is induced by the inclusions $\mathrm{Sets}(H^{\bullet+1}, X_f)^{N_\lambda} \to 
\mathrm{Sets}(H^{\bullet+1}, X_f)$.

Now suppose that $X$ is a discrete $H$-spectrum. As in 
\cite[Sections 2.4, 3.2]{joint}, there is a cosimplicial spectrum 
$\Gamma_H^\bullet X$, where 
for each $n \geq 0$, the spectrum of $n$-cosimplices of 
$\Gamma_H^\bullet X$ satisfies the isomorphism
\[(\Gamma_H^\bullet X)^n 
\cong \colim_{U \vartriangleleft_o H^n} \prod_{H^n/U} 
\mspace{1mu} X,\] where $H^n$ is the $n$-fold cartesian product of copies of 
$H$ ($H^0$ is the trivial group $\{e\}$) and the 
colimit is over all the open normal subgroups of $H^n$. By \cite[Theorem 3.2.1]{joint} and 
\cite[page 330, Remark 7.5]{cts}, if 
$H = G$, a closed subgroup of $\G$, then 
\begin{equation}\label{Godement}\zig
X^{hG} \simeq \holim_\Delta \Gamma^\bullet_G \,X_\mathtt{fib},\end{equation}
where $X_\mathtt{fib}$ is any discrete $G$-spectrum that is fibrant as a spectrum and 
is equipped with a $G$-equivariant map 
$X \xrightarrow{\,\simeq\,} X_\mathtt{fib}$ that is a weak equivalence of spectra. 

\section{Step iii: the role of $V(1)$ in the 
spectral sequences $\{{^j\mspace{-3mu}E_r^{\ast,\ast}}\}$}\label{step3}
  
Now we focus on understanding the part played by $V(1)$ in 
spectral sequence $\{{^j\mspace{-3mu}E_r^{\ast,\ast}}\}$, where $j \geq 0$ (and $G$ is any closed subgroup of $\G$). Let 
\[\mathcal{O} = \{p^m\mathbb{Z}_p\}_{m \geq 0},\] where each $p^m\mathbb{Z}_p$ is the open normal subgroup of $\G$ that corresponds to 
\[(p^m\mathbb{Z}_p) \times \{e\} \vartriangleleft_o \padics \times C_{p-1}.\] 
In the introduction, we noted that $K(KU_p) \wedge v_2^{-1}V(1)$ is realized by 
the discrete $\G$-spectrum $C^\mathrm{dis}_p$, which we can now define:
\[C^\mathrm{dis}_p := \colim_{j \geq 0} \bigl(((K(KU_p) \wedge \Sigma^{-jd}V(1))_f)^\mathrm{dis}_\mathcal{O}\bigr).\] By \cite[Remark 1.20, (8.1)]{padicspectra2}, spectral sequence 
$\{E_r^{\ast,\ast}\}$ is the homotopy spectral sequence for 
$\holim_\Delta \Gamma_G^\bullet C^\mathrm{dis}_p$. In Section \ref{stepi}, 
we noted that there is the isomorphism
\[\{E_r^{\ast,\ast}\} \cong \colim_{j \geq 0} \{{^j\mspace{-3mu}E_r^{\ast,\ast}}\}\] of 
spectral sequences; for each $j$, $\{{^j\mspace{-3mu}E_r^{\ast,\ast}}\}$ is the homotopy spectral sequence for 
\[\holim_\Delta \Gamma_G^\bullet \bigl(((K(KU_p) \wedge \Sigma^{-jd}V(1))_f)^\mathrm{dis}_\mathcal{O}\bigr).\] 

Fix any $j \geq 0$. To increase readability and when the additional intuition carried by the original notation is not needed, we will sometimes use the abbreviation 
\[\mathbb{K}_j : = K(KU_p) \wedge \Sigma^{-jd}V(1).\] Since the fibrant replacement morphism $\mathbb{K}_j \to (\mathbb{K}_j)_f$ is a weak equivalence of spectra that is $\G$-equivariant, the induced map $(\mathbb{K}_j)^\mathrm{dis}_\mathcal{O} \to ((\mathbb{K}_j)_f)^\mathrm{dis}_\mathcal{O}$ is a weak equivalence that is $\G$-equivariant, by \cite[Remark 1.20, paragraph after (8.4)]{padicspectra2}. 
If $X$ is a discrete $G$-spectrum, then for each $n \geq 0$, the spectrum of 
$n$-cosimplices of $\Gamma^\bullet_G X$ is obtained by applying iteratively $n$ times to $X$ a functor that preserves weak equivalences of spectra, by 
\cite[Lemma 2.4.1]{joint}. Thus, the induced morphism  
\[\Gamma_G^\bullet (\mathbb{K}_j)^\mathrm{dis}_\mathcal{O}\xrightarrow{\,\simeq\,} \Gamma_G^\bullet \bigl(((\mathbb{K}_j)_f)^\mathrm{dis}_\mathcal{O}\bigr)\] is an objectwise weak equivalence of cosimplicial spectra, so spectral sequence $\{{^j\mspace{-3mu}E_r^{\ast,\ast}}\}$ is isomorphic 
to the homotopy spectral sequence for 
\[\holim_\Delta\Gamma_G^\bullet (\mathbb{K}_j)^\mathrm{dis}_\mathcal{O}.\] Hence, we shift our focus to this latter spectral sequence.

For each $n \geq 0$, 
the spectrum of $n$-cosimplices of $\Gamma_G^\bullet (\mathbb{K}_j)^\mathrm{dis}_\mathcal{O}$ satisfies  
\[
\bigl(\Gamma_G^\bullet (\mathbb{K}_j)^\mathrm{dis}_\mathcal{O}\bigr)^{\mspace{-3mu}n} \cong \colim_{U \vartriangleleft_o G^n} \prod_{G^n/U}  
\colim_{m \geq 0} \holim_\Delta \mathrm{Sets}((\G)^{\bullet + 1}, 
(\mathbb{K}_j)_f)^{p^m\mathbb{Z}_p}.
\] Now choose any $m \geq 0$. 
Again at the level of $n$-cosimplices, we have
\begin{align*}
\bigl(\mathrm{Sets}((\G)^{\bullet + 1}, 
(\mathbb{K}_j)_f)^{p^m\mathbb{Z}_p}\bigr)^{\mspace{-3mu}n} 
& \cong \prod_{\scriptstyle{\G/(p^m\mathbb{Z}_p)}} \prod_{(\G)^n} \,(K(KU_p) \wedge \Sigma^{-jd}V(1))_f \\ & 
\simeq \bigl(\mathrm{Sets}((\G)^{\bullet + 1}, 
(K(KU_p))_f)^{p^m\mathbb{Z}_p}\bigr)^{\mspace{-3mu}n} \wedge \Sigma^{-jd}V(1),
\end{align*}
where the isomorphism is as in \cite[proof of Lemma 2.1]{padicspectra2} and the second step applies the fact that smashing with a finite spectrum commutes with any product. If $Z^\bullet \: \Delta \to Sp^\Sigma$ is a cosimplicial spectrum and $Z'$ is any spectrum, then there is the functor \[(-) \wedge Z' \: Sp^\Sigma \to Sp^\Sigma, \ \ \ Y \mapsto Y \wedge Z',\] and we let $Z^\bullet \wedge Z'$ denote the cosimplicial spectrum $((-) \wedge Z') \circ Z^\bullet$. Then we have 
\begin{align*}
\holim_\Delta \mathrm{Sets}((\G&)^{\bullet + 1}, 
(\mathbb{K}_j)_f)^{p^m\mathbb{Z}_p}\\ &\simeq 
\holim_\Delta (\mathrm{Sets}((\G)^{\bullet + 1}, 
(K(KU_p))_f)^{p^m\mathbb{Z}_p} \wedge \Sigma^{-jd}V(1))_f \\
& \simeq \bigl(\holim_\Delta \mathrm{Sets}((\G)^{\bullet + 1}, 
(K(KU_p))_f)^{p^m\mathbb{Z}_p}\bigr) \wedge \Sigma^{-jd}V(1),\end{align*} 
where the last step is because $\Sigma^{-jd}V(1)$ is a finite spectrum.

Our last conclusion implies that for each $n \geq 0$, we have 
\begin{align*}
\bigl(\Gamma_G^\bullet &(\mathbb{K}_j)^\mathrm{dis}_\mathcal{O}\bigr)^{\mspace{-3mu}n}\\
& \simeq \colim_{U \vartriangleleft_o G^n} \prod_{G^n/U}  
\colim_{m \geq 0} \bigl(\bigl(\holim_\Delta \mathrm{Sets}((\G)^{\bullet + 1}, 
(K(KU_p))_f)^{p^m\mathbb{Z}_p}\bigr) \wedge \Sigma^{-jd}V(1)\bigr)_{\mspace{-4mu}f} \\
& \simeq  
\bigl(\mspace{1.3mu}\colim_{U \vartriangleleft_o G^n} \prod_{G^n/U}  
\colim_{m \geq 0} \holim_\Delta \mathrm{Sets}((\G)^{\bullet + 1}, 
(K(KU_p))_f)^{p^m\mathbb{Z}_p}\bigr) \wedge \Sigma^{-jd}V(1) \\
& \cong \bigl(\Gamma_G^\bullet (K(KU_p))^\mathrm{dis}_\mathcal{O}\bigr)^{\mspace{-3mu}n} 
\wedge \Sigma^{-jd}V(1),
\end{align*} where the second step uses that the smash product commutes with 
colimits and finite products (which are weakly equivalent to finite coproducts). This shows that there is a zigzag of objectwise 
weak equivalences between the following two cosimplicial spectra:
\begin{equation}\label{reedy}\zig
\Gamma_G^\bullet \bigl((K(KU_p) \wedge \Sigma^{-jd}V(1))^\mathrm{dis}_\mathcal{O}\bigr) \simeq 
\bigl(\bigl(\Gamma_G^\bullet (K(KU_p))^\mathrm{dis}_\mathcal{O}\bigr) 
\wedge \Sigma^{-jd}V(1)\bigr)_{\mspace{-3mu}f}\,.\end{equation} 
If $Z^\bullet$ is a 
cosimplicial spectrum that is objectwise fibrant, we let $\mathfrak{hss}(Z^\bullet)$ denote the associated homotopy spectral sequence. We have shown that there 
are isomorphisms 
\begin{align*}
\{{^j\mspace{-3mu}E_r^{\ast,\ast}}\} & \cong 
\mathfrak{hss}\bigl(\Gamma_G^\bullet \bigl((K(KU_p)\wedge \Sigma^{-jd}V(1))^\mathrm{dis}_\mathcal{O}\bigr)\bigr)\\ & \cong \mathfrak{hss}\bigl(\bigl(\bigl(\Gamma_G^\bullet (K(KU_p))^\mathrm{dis}_\mathcal{O}\bigr) 
\wedge \Sigma^{-jd}V(1)\bigr)_{\mspace{-3mu}f}\bigr)\end{align*} of 
spectral sequences; the first isomorphism was obtained earlier in this section 
and the second one is by (\ref{reedy}), which also yields the following 
result. 

\begin{Thm}\label{hfp}
Let $p \geq 5$. If $G$ is a closed subgroup of $\G$ and $j \geq 0$, then  
\[(K(KU_p) \wedge \Sigma^{-jd}V(1))^{hG} 
\simeq ((K(KU_p))^\mathrm{dis}_\mathcal{O})^{hG} 
\wedge \Sigma^{-jd}V(1).\]
\end{Thm}
\begin{proof} 
We have
\begin{align*}
(K(KU_p) \wedge \Sigma^{-jd}V(1))^{hG} & := 
((K(KU_p) \wedge \Sigma^{-jd}V(1))^\mathrm{dis}_\mathcal{O})^{hG}\\
& \simeq
\holim_\Delta \Gamma_G^\bullet \bigl((K(KU_p) \wedge \Sigma^{-jd}V(1))^\mathrm{dis}_\mathcal{O}\bigr)\\
& \simeq \holim_\Delta \bigl(\bigl(\Gamma_G^\bullet (K(KU_p))^\mathrm{dis}_\mathcal{O}\bigr) 
\wedge \Sigma^{-jd}V(1)\bigr)_{\mspace{-3mu}f}\\
& \simeq \bigl(\holim_\Delta \Gamma_G^\bullet (K(KU_p))^\mathrm{dis}_\mathcal{O}\bigr) 
\wedge \Sigma^{-jd}V(1),\end{align*} where each step is justified 
by \cite[end of Section 1.2]{padicspectra2}, (\ref{Godement}), 
(\ref{reedy}), and the fact that $\Sigma^{-jd}V(1)$ is 
a finite spectrum, respectively, and the last expression above is equivalent to 
the right-hand side 
in the desired result (again, by (\ref{Godement})).
\end{proof}

\section {Step iv: each spectral sequence is one of $P(v_2)$-modules}\label{step4}
In this section, $j \geq 0$ and, as usual, $G$ is any closed subgroup of $\G$. 

Since $p \geq 5$, $V(1)$ is a homotopy commutative and homotopy associative ring spectrum. Then by Theorem \ref{hfp},
\[\pi_\ast\bigl((K(KU_p) \wedge \Sigma^{-jd}V(1))^{hG}\bigr) 
\cong \pi_\ast(((K(KU_p))^\mathrm{dis}_\mathcal{O})^{hG} 
\wedge \Sigma^{-jd}V(1))\] is a right $\pi_\ast(V(1))$-module, and 
hence, it is a $P(v_2)$-module. This observation suggests that 
spectral sequence \[\mathfrak{HS}_j := \mathfrak{hss}\bigl(\bigl(\bigl(\Gamma_G^\bullet (K(KU_p))^\mathrm{dis}_\mathcal{O}\bigr) 
\wedge \Sigma^{-jd}V(1)\bigr)_{\mspace{-3mu}f}\bigr)\] is one of 
$P(v_2)$-modules, and now we show that this is the case. 

If $Z^\bullet$ is a cosimplicial spectrum, let $\prod^\ast \negthinspace Z^\bullet$ be its 
cosimplicial replacement. Also, 
let \[C^\bullet := \Gamma_G^\bullet (K(KU_p))^\mathrm{dis}_\mathcal{O},\] 
so that 
\[\holim_\Delta \bigl(\bigl(\Gamma_G^\bullet (K(KU_p))^\mathrm{dis}_\mathcal{O}\bigr) 
\wedge \Sigma^{-jd}V(1)\bigr)_{\mspace{-3mu}f} = \mathrm{Tot}\bigl(\textstyle{\prod^\ast} 
(C^\bullet \wedge \Sigma^{-jd}V(1))_f\bigr).\] For each $l \geq 0$, let 
\begin{equation}\label{fiber}\zig
F_l \to \mathrm{Tot}_l\bigl(\textstyle{\prod}^\ast (C^\bullet \wedge \Sigma^{-jd}V(1))_f\bigr) \to 
 \mathrm{Tot}_{l-1}\bigl(\prod^\ast (C^\bullet \wedge \Sigma^{-jd}V(1))_f\bigr)
 \end{equation}
 be a homotopy fiber sequence (when $l = 0$, the last term above is $\ast$, the trivial 
 spectrum) and to conserve space, let 
\[\mathbb{T}_l(-):= \mathrm{Tot}_l(\textstyle{\prod}^\ast(-)) \ \ \ \text{and} \ \ \ C^\bullet_j := (C^\bullet \wedge \Sigma^{-jd}V(1))_f.\] Then 
$\mathfrak{HS}_j$ is the spectral sequence obtained from the 
 exact couple formed from the long exact sequences 
\[{\cdots \to \pi_t(F_l) \to \pi_t(\mathbb{T}_l(C^\bullet_j)) \to 
 \pi_t(\mathbb{T}_{l-1}(C^\bullet_j)) \to \pi_{t-1}(F_l) \to \cdots}\] 
associated to the above homotopy fiber sequences.  

As done earlier, we now exploit the fact that smashing with a finite spectrum 
commutes with products and homotopy limits. Notice that for each $n \geq 0$, 
\begin{align*}
\bigl(\textstyle{\prod}^\ast (C^\bullet \wedge \Sigma^{-jd}V(1))_f\bigr)^{\mspace{-3mu}n} 
& = \textstyle{\prod}_{\scriptscriptstyle{\{[j_0] \to \cdots \to [j_n]\}}} (C^{j_n} \wedge \Sigma^{-jd}V(1))_f \\
& \simeq \bigl(\textstyle{\prod}_{\scriptscriptstyle{\{[j_0] \to \cdots \to [j_n]\}}} C^{j_n}\bigr) \wedge \Sigma^{-jd}V(1) 
\\
& = (\textstyle{\prod}^* C^\bullet)^n \wedge \Sigma^{-jd}V(1),\end{align*} where 
the middle two products are indexed over all length $n$ compositions in the 
category $\Delta$, so that 
\[\textstyle{\prod}^\ast (C^\bullet \wedge \Sigma^{-jd}V(1))_f 
\simeq \bigl((\textstyle{\prod}^* C^\bullet) \wedge \Sigma^{-jd}V(1)\bigr)_{\mspace{-3mu}f},\] which depicts 
a zigzag of objectwise weak equivalences 
between cosimplicial spectra. Then for each $l \geq 0$, with $\Delta^{(l)}$ 
equal to the full subcategory of $\Delta$ consisting of objects of cardinality 
less than $l+2$, and -- given a cosimplicial spectrum $Z^\bullet$ -- using 
$\holim_{\Delta^{(l)}} Z^\bullet$ to denote 
$\holim_{\Delta^{(l)}} \bigl(\Delta^{(l)} \hookrightarrow \Delta \xrightarrow{Z^\bullet} 
Sp^\Sigma\bigr),$ we have
\begin{align*}
\mathrm{Tot}_l\bigl(\textstyle{\prod}^\ast (C^\bullet &\wedge \Sigma^{-jd}V(1))_f\bigr) 
\\ &\simeq \holim_{\Delta^{(l)}} \textstyle{\prod}^\ast (C^\bullet \wedge \Sigma^{-jd}V(1))_f \simeq \displaystyle{\holim_{\Delta^{(l)}}} \bigl((\textstyle{\prod}^\ast C^\bullet) \wedge \Sigma^{-jd}V(1)\bigr)_{\mspace{-3mu}f}\\ & \simeq \bigl(\holim_{\Delta^{(l)}} \textstyle{\prod}^\ast C^\bullet\bigr) \wedge \Sigma^{-jd}V(1) \simeq \mathrm{Tot}_l(\textstyle{\prod}^\ast C^\bullet) \wedge \Sigma^{-jd}V(1),
\end{align*} where the first and last steps are by \cite[Proposition 3.10]{cubical}. Thus, in the 
stable homotopy category, for $l \geq 0$, we can regard the homotopy fiber sequence in 
(\ref{fiber}) as having 
the form
\begin{equation}\label{fiber2}\zig
F_l \to \mathrm{Tot}_l(\textstyle{\prod}^\ast C^\bullet) \wedge \Sigma^{-jd}V(1) \to 
\mathrm{Tot}_{l-1}(\textstyle{\prod}^\ast C^\bullet) \wedge \Sigma^{-jd}V(1).
\end{equation}
 
Since the stable model structure on $Sp^\Sigma$ is proper \cite[Theorem 5.5.2]{HSS}, by \cite[Remark 19.1.6, Propositions 13.4.4 and 19.5.3]{hirschhorn}, we 
can regard a homotopy fiber as a homotopy limit. For each $l \geq 0$, let 
\[\widehat{F}_l 
\xrightarrow{\gamma_l} \mathrm{Tot}_l(\textstyle{\prod}^\ast C^\bullet) 
\xrightarrow{\alpha_l} \mathrm{Tot}_{l-1}(\textstyle{\prod}^\ast C^\bullet) 
\xrightarrow{\beta_l} \Sigma\widehat{F}_l \] be a 
homotopy fiber sequence (our names for the maps follow \cite[(5.29)]{Thomason}): by an application of $(-) \wedge \Sigma^{-jd}V(1)$, 
we obtain the homotopy fiber sequence 
\[\widehat{F}_l \wedge \Sigma^{-jd}V(1)
\xrightarrow{\gamma_l \wedge 1} \mathrm{Tot}_l(\textstyle{\prod}^\ast C^\bullet)  \wedge \Sigma^{-jd}V(1) \xrightarrow{\alpha_l \wedge 1} \mathrm{Tot}_{l-1}(\textstyle{\prod}^\ast C^\bullet) \wedge \Sigma^{-jd}V(1).\] By comparing 
this fiber sequence with (\ref{fiber2}), another 
application of commuting a homotopy limit with smashing with a finite spectrum 
yields
\[F_l \simeq \widehat{F}_l \wedge \Sigma^{-jd}V(1), \ \ \ l \geq 0.\] 
It follows that $\mathfrak{HS}_j$ is the spectral sequence obtained from the exact couple formed from the long exact sequences 
\begin{align*}
{\cdots} \displaystyle{}\to \pi_\ast(&\widehat{F}_l \wedge \Sigma^{-jd}V(1)) \xrightarrow{(\gamma_l \wedge 1)_\ast} \pi_\ast(\mathbb{T}_l(C^\bullet) 
\wedge \Sigma^{-jd}V(1)) \,{\scriptstyle{-\cdots}}\\ & \xrightarrow{(\alpha_l \wedge 1)_\ast}
 \pi_\ast(\mathbb{T}_{l-1}(C^\bullet) 
\wedge \Sigma^{-jd}V(1)) \xrightarrow{(\beta_l \wedge 1)_\ast} \pi_{\ast-1}(\widehat{F}_l \wedge \Sigma^{-jd}V(1)) \to \,{\cdots}
\end{align*} (the top row ends with a morphism that is continued in the bottom row), 
where $l~\geq~0$. 
As recalled earlier, $V(1)$ is a homotopy commutative and homotopy associative ring spectrum, so that this long exact sequence is in the category 
of $P(v_2)$-modules. Thus, the associated exact couple and, 
consequently, spectral sequence $\mathfrak{HS}_j$ live in the category of 
$P(v_2)$-modules. 

\section{Step v: the $P(v_2)$-module spectral sequences give a direct sum}\label{step5}

As usual, $G$ is any closed subgroup of $\G$, and 
$\mathrm{Mod}_{P(v_2)}$ is the category of $P(v_2)$-modules. 
We recall from Section \ref{stepi} that there is the isomorphism
\[\pi_{\ast}\bigl((K(KU_p) \wedge v_2^{-1}V(1))^{ hG}\bigr) \cong 
\colim_{j \geq 0} \pi_{\ast}\bigl((K(KU_p) \wedge \Sigma^{-jd}V(1))^{ hG}\bigr),\] 
where the right-hand side is the middle term in the colimit 
\[\colim_{j \geq 0} \bigl(0 \to {^j\mspace{-3mu}E_2^{1,\ast+1}} \to 
\pi_{\ast}\bigl((K(KU_p) \wedge \Sigma^{-jd}V(1))^{ hG}\bigr) \to {^j\mspace{-3mu}E_2^{0,\ast}} \to 0\bigr)\] of short 
exact sequences. For each $j \geq 0$, 
$\mathfrak{HS}_j$ is a spectral sequence in $\mathrm{Mod}_{P(v_2)}$ 
and since it is isomorphic to spectral sequence $\{{^j{\mspace{-3mu}E_r^{\ast,\ast}}}\}$, the associated short exact sequence (displayed above, inside the parentheses) 
is in $\mathrm{Mod}_{P(v_2)}$. 
It will be helpful to write out this short exact sequence explicitly: omitting the trivial terms on the ends and letting $\mathsf{K}$ denote $K(KU_p)$, this sequence 
of $P(v_2)$-modules has the form
\[H^1_c(G, \pi_{\ast+1}(\mathsf{K} \wedge \Sigma^{-jd}V(1))) \to 
\pi_\ast((\mathsf{K}^\mathrm{dis}_\mathcal{O})^{hG} \wedge \Sigma^{-jd}V(1)) 
\to (\pi_\ast(\mathsf{K} \wedge \Sigma^{-jd}V(1)))^G,\] where the middle term 
resulted from applying Theorem \ref{hfp}. 

If $Z$ is any spectrum, then the diagram $\{\pi_\ast(Z \wedge \Sigma^{-jd}V(1))\}_{j \geq 0}$ 
is in $\mathrm{Mod}_{P(v_2)}$, so that 
the isomorphism \[\pi_\ast(Z \wedge v_2^{-1}V(1)) \cong 
\colim_{j \geq 0} \pi_\ast(Z \wedge \Sigma^{-jd}V(1))\] is in the category of $P(v_2^{\pm 1})$-modules (for example, see \cite[Corollary 1.2]{publishedjems}). 
The direct system of spectra $\{\Sigma^{-jd}V(1)\}_{j \geq 0}$ induces 
a direct system 
\[\bigl\{\bigl(\bigl(\Gamma^\bullet_G(K(KU_p))^\mathrm{dis}_\mathcal{O}\bigr) 
\wedge \Sigma^{-jd}V(1)\bigr)_{\mspace{-3mu}f}\bigr\}_{\mspace{-3mu}j \geq 0}\] of 
cosimplicial spectra, and hence, 
a direct system $\{\mathfrak{HS}_j\}_{j \geq 0}$ of homotopy spectral sequences. 
Thus, there is the direct system 
\[\bigl\{\pi_\ast\bigl(\bigl(\bigl(\Gamma^\ast_G(K(KU_p))^\mathrm{dis}_\mathcal{O}\bigr) 
\wedge \Sigma^{-jd}V(1)\bigr)_{\mspace{-3mu}f}\bigr)\bigr\}_{\mspace{-3mu}j \geq 0}\] of associated cochain complexes in $\mathrm{Mod}_{P(v_2)}$, the cohomology of which induces 
the direct system 
\[\bigl\{H^s_c(G, \pi_\ast(K(KU_p) \wedge \Sigma^{-jd}V(1)))\bigr\}_{j \geq 0}\] 
in $\mathrm{Mod}_{P(v_2)}$, for $s = 0, 1$. Therefore, the diagram
\[\bigl\{0 \to {^j\mspace{-3mu}E_2^{1,\ast+1}} \to 
\pi_{\ast}(((K(KU_p))^\mathrm{dis}_\mathcal{O})^{hG} \wedge \Sigma^{-jd}V(1)) \to {^j\mspace{-3mu}E_2^{0,\ast}} \to 0\bigr\}_{\mspace{-3mu}j \geq 0}\] of short exact sequences 
is in $\mathrm{Mod}_{P(v_2)}$, so that the exact sequence
\[0 \to \colim_{j \geq 0} {^j\mspace{-3mu}E_2^{1,\ast+1}} \to 
\colim_{j \geq 0} \pi_{\ast}(((K(KU_p))^\mathrm{dis}_\mathcal{O})^{hG} \wedge \Sigma^{-jd}V(1)) \to\colim_{j \geq 0} {^j\mspace{-3mu}E_2^{0,\ast}} \to 0\] is in the category of $P(v_2^{\pm1})$-modules, where the isomorphisms
\[\colim_{j \geq 0} {^j\mspace{-3mu}E_2^{s,\ast}} \cong H^s\negthinspace 
\Bigl[\colim_{j \geq 0} 
\pi_\ast\bigl(\bigl(\bigl(\Gamma^\ast_G(K(KU_p))^\mathrm{dis}_\mathcal{O}\bigr) 
\wedge \Sigma^{-jd}V(1)\bigr)_{\mspace{-3mu}f}\bigr)\Bigr], \ \ \ s = 0, 1,\] show that the 
two outer nontrivial terms in the exact sequence are indeed modules over $P(v_2^{\pm1})$. In particular, in every degree $t$, the sequence is one of 
$\mathbb{F}_p$-modules and is split exact, giving 
\begin{align*}
\pi_{t}&\bigl((K(KU_p) \wedge v_2^{-1}V(1))^{hG}\bigr)\\ & \cong 
\bigl(\colim_{j \geq 0} H^1_c(G, \pi_{t+1}(K(KU_p) \wedge \Sigma^{-jd}V(1)))\bigr) 
\oplus \bigl(\pi_{t}(K(KU_p) \wedge v_2^{-1}V(1))\bigr)^{\mspace{-3mu}G},
\end{align*} 
an isomorphism of $\mathbb{F}_p$-modules. 

\section{Step vi: simplifying $H^1_c(G, \pi_\ast(K(KU_p) \wedge V(1))[v_2^{-1}])$}\label{step6}
Now we work on reducing the first summand in the direct sum obtained at the end of the previous section to a more familiar object. Fix $j \geq 0$ and $t \in \zee$, and recall that 
\[\pi_t(\mathbb{K}_j) = \pi_t(K(KU_p) \wedge \Sigma^{-jd}V(1))\] is a finite abelian group (this fact is explained in \cite[Section 1.2]{padicspectra2}; the author did not play a role in the hard work behind the explanation, which was done by others, as 
noted by the references in [ibid.]) and, as a unitary $\mathbb{F}_p$-module, it is a $p$-torsion group (that is, $pm = 0$, for every element $m$). Notice that 
\begin{align*}
H^1_c(G, \pi_{t}(K(KU_p) \wedge \Sigma^{-jd}V(1))) 
& \cong H^1_c(\G, \mathrm{Coind}^{\G}_G\mspace{-4mu}(\pi_{t}(\mathbb{K}_j)))\\
& \cong \colim_{N \vartriangleleft_o \G} H^1_c(\G, C^N_{(t,j)}),\end{align*} where the first isomorphism is by Shapiro's Lemma, the 
second one is by \cite[Proposition 6.10.4, (a)]{Ribes} -- with 
\[C^N_{(t,j)} : = 
\mathrm{Coind}^{\G/N}_{GN/N}((\pi_{t}(\mathbb{K}_j))^{N \cap G}),\] and 
each $C^N_{(t,j)}$ is a $\G/N$-module (by definition), which makes $C^N_{(t,j)}$ 
a discrete $\G$-module 
via the projection $\G \to \G/N$. 

Let $N$ be fixed. As a set, $C^N_{(t,j)}$ 
is finite and, for every element $f$ in this abelian group, $pf = 0$. This last fact -- together with $p-1$ and $p$ being relatively prime -- implies that the cohomology 
$H^\ast(C_{p-1}, C^N_{(t,j)})$ for the $C_{p-1}$-module $C^N_{(t,j)}$ (by restriction of the $\G$-action) vanishes in positive degrees, so that in the 
Lyndon-Hochschild-Serre spectral sequence 
\[E_2^{p,q} = H^p_c\bigl(\padics, H^q(C_{p-1}, C^N_{(t,j)})\bigr) \Longrightarrow 
H^{p+q}_c(\G, C^N_{(t,j)}),\] we have
\[E_2^{p,q} = \begin{cases}
H^p_c\mspace{-2mu}\bigl(\padics, \bigl(C^N_{(t,j)}\bigr)^{\mspace{-3mu}C_{p-1}}\bigr), & q = 0;\\
0, & q > 0.\end{cases}\] 
This gives 
\[H^1_c(\G, C^N_{(t,j)}) \cong H^1_c\mspace{-2mu}\bigl(\padics, \bigl(C^N_{(t,j)}\bigr)^{\mspace{-3mu}C_{p-1}}\bigr) \cong H^1\mspace{-2mu}\bigl(\zee, \bigl(C^N_{(t,j)}\bigr)^{\mspace{-3mu}C_{p-1}}\bigr) \cong \bigl(\bigl(C^N_{(t,j)}\bigr)^{\mspace{-3mu}C_{p-1}}\bigr)_{_{\mspace{-3mu}\scriptstyle{\zee}}},\]
where the third expression above is 
a non-continuous cohomology group and the second isomorphism is because 
$\bigl(C^N_{(t,j)}\bigr)^{\mspace{-3mu}C_{p-1}}$ is finite and $p$-torsion 
(for example, see \cite[Example 4.6, Lemma 4.7]{LinnellSchick}).

Now we put the pieces together as $j$ varies. Given a group $K$, let 
$\zee[K]$-$\mathrm{Mod}$ be the category of $K$-modules, and let $\mathrm{Ab}$ denote the category of abelian groups. Also, given mathematical expressions $A$ and $B$, notation of the form \[A \overset{K/e/L}{\cong} B \ \ \ \text{or} \ \ \  A \overset{e/K/L}{\cong} B\] means that (a) in $\mathrm{Ab}$, $A \cong B$; (b) in expression 
$A$, any colimits are in $\zee[K]$-$\mathrm{Mod}$ or $\mathrm{Ab}$ (signified by ``$e$" in ``$e/K/L$"), 
respectively, but these colimits can be formed in 
$\mathrm{Ab}$ or $\zee[K]$-$\mathrm{Mod}$, respectively, since the forgetful functor 
$\zee[K]\text{-}\mathrm{Mod} \to \mathrm{Ab}$ is a left adjoint; (c) part (b) explains 
the commuting of any colimits with the evident functor and this commuting 
underlies the 
isomorphism $A \cong B$; and (d) $L$ denotes a group, and in $B$, any colimits are 
in $\zee[L]$-$\mathrm{Mod}$, by which we mean $\mathrm{Ab}$, when 
$L$ is ``$e$." (To avoid any confusion, we note that if $K = \zee$, then 
$\zee[K]$-$\mathrm{Mod}$ means ${\zee[\zee]}$-$\mathrm{Mod}$.) We have
\begin{align*}
\colim_{j \geq 0} H^1_c(&G, \pi_{t}(K(KU_p) \wedge \Sigma^{-jd}V(1))) 
\cong \colim_{N \vartriangleleft_o \G} \colim_{j \geq 0} \bigl(\bigl(C^N_{(t,j)}\bigr)^{\mspace{-3mu}C_{p-1}}\bigr)_{_{\mspace{-3mu}\scriptstyle{\zee}}}\\
& \overset{e/e/\zee}{\cong} \Bigl(\,\colim_{N \vartriangleleft_o \G} \colim_{j \geq 0} \bigl(C^N_{(t,j)}\bigr)^{\mspace{-3mu}C_{p-1}}\Bigr)_{_{\mspace{-6mu}\scriptstyle{\zee}}} 
\overset{\zee/e/e}{\cong} \Bigl(\Bigl(\,\colim_{N \vartriangleleft_o \G} \colim_{j \geq 0} C^N_{(t,j)}\Bigr)^{\mspace{-5mu}C_{p-1}}\Bigr)_{_{\mspace{-6mu}\scriptstyle{\zee}}}\, .\end{align*} 

Again, let $N \vartriangleleft_o \G$ be fixed and, as is standard, given $A \in 
\zee[GN/N]\text{-}\mathrm{Mod}$, let 
\[\mathrm{Ind}^{\G/N}_{GN/N}(A) = \mathbb{Z}[\G/N] \otimes_{\mathbb{Z}[GN/N]} A,\] 
and set 
\[\mathbb{P}_j := (\pi_{t}(\mathbb{K}_j))^{N \cap G}.\]
Then there are isomorphisms  
\[
C^N_{(t,j)} \cong \mathrm{Hom}_{\zee[GN/N]\text{-}\mathrm{Mod}}(\mathbb{Z}[\G/N], 
\mathbb{P}_j)
\cong \mathrm{Ind}^{\G/N}_{GN/N}(\mathbb{P}_j)\] of $\zee[\G/N]$-modules, since $\G/N$ is finite (for example, see 
\cite[proof of Proposition 6.10.4]{Ribes}) and because 
$(\G/N)/(GN/N) \cong \G/GN$ is finite \cite[Proposition 5.9]{KennethBrown}, respectively. 
Hence, there are the following isomorphisms of $\zee[\G/N$]-modules (in the 
first use below of the ``$\,\scriptstyle{\overset{e/K/L}{\cong}}\,$" notation, part (c) of its meaning 
does not apply):
\begin{align*}
\colim_{j \geq 0}& \,C^N_{(t,j)} \overset{e/(\G/N)/(\G/N)}{\cong} \mspace{-2mu}
\colim_{j \geq 0} \mathrm{Ind}^{\G/N}_{GN/N}(\mathbb{P}_j) \mspace{-2mu}
\overset{e/(\G/N)/(GN/N)}{\cong} \mspace{-2mu} \mathrm{Ind}^{\G/N}_{GN/N}(\colim_{j \geq 0} \mathbb{P}_j)\\
& 
\cong \mathrm{Coind}^{\G/N}_{GN/N}(\colim_{j \geq 0} \mathbb{P}_j)
\cong \mathrm{Coind}^{\G/N}_{GN/N}((\pi_{t}(K(KU_p) \wedge v_2^{-1}V(1)))^{N \cap G}).\end{align*} These four isomorphisms are of $\zee[\G]$-modules (via the 
projection $\G \to \G/N$) and, by \cite[Proposition 6.10.4, (a)]{Ribes}, we conclude 
that 
\[
\colim_{j \geq 0} H^1_c(G,\pi_{t}(K(KU_p) \wedge \Sigma^{-jd}V(1))) 
\negthinspace \cong \negthinspace 
\bigl(\mspace{-2mu}\bigl(\mathrm{Coind}^{\G}_{G}\mspace{-4mu}(\pi_{t}(K(KU_p) \wedge v_2^{-1}V(1)))\bigr)^{\mspace{-2mu}C_{p-1}}\bigr)_{_{\mspace{-3mu}\scriptstyle{\zee}}},
\] completing the proof of Theorem \ref{directsum}.

In case it is easier to compute $\colim_{j \geq 0} H^1_c(G,\pi_{t}(K(KU_p) \wedge \Sigma^{-jd}V(1)))$ by not restricting the 
$\padics$-action to the $\zee$-action, as done on the right-hand side in the last isomorphism above, we take another look at each 
$H^1_c\mspace{-2mu}\bigl(\padics, \bigl(C^N_{(t,j)}\bigr)^{\mspace{-3mu}C_{p-1}}\bigr)$ to obtain the following result, which is the content of Remark \ref{ctscoinvariants}. Here (as in the remark), $\padics$ is regarded as having the 
trivial $\padics$-group action. 

\begin{Lem}\label{not_restricted} 
When $p \geq 5$, $G$ is any closed subgroup of $\G$, and $t \in \zee$, there is an isomorphism 
\begin{align*}
H^1_c(G, & \,\pi_t(K(KU_p) \wedge v_2^{-1}V(1)))\\ 
& \cong 
\padics {\otimes}_{\padics[[\padics]]} 
\bigl(\bigl(\mathrm{Coind}^{\G}_{G}\mspace{-3mu}(\pi_{t}(K(KU_p) \wedge v_2^{-1}V(1)))\bigr)^{\mspace{-3mu}C_{p-1}}\bigr).
\end{align*}
\end{Lem}
\begin{proof}
Notice that every coefficient group $\bigl(C^N_{(t,j)}\bigr)^{\mspace{-3mu}C_{p-1}}$ is a finite discrete $p$-torsion $\padics[[\padics]]$-module. It is standard that there is a projective resolution
\[0 \to \padics[[\padics]] 
\xrightarrow{\,\tau\,} \padics[[\padics]] \to \padics \to 0\] (for example, see \cite[proof of Proposition 6]{minicourse} for any omitted details) that can be used to compute the cohomology group (for more 
information about this, see \cite[Section 3.2]{Symonds}). 
Thus, the cohomology group is the cohomology of the complex obtained by applying 
the functor $\mathrm{Hom}_{\padics[[\padics]]}^c(-,\bigl(C^N_{(t,j)}\bigr)^{\mspace{-3mu}C_{p-1}})$ of continuous module homomorphisms 
to this resolution: we obtain that
\begin{align*}
H^1_c\mspace{-2mu}\bigl(\padics, \bigl(C^N_{(t,j)}\bigr)^{\mspace{-3mu}C_{p-1}}\bigr) & \cong \faktor{\scriptstyle{(C^N_{(t,j)})^{\mspace{-1mu}C_{p-1}}\mspace{-8mu}}}{\scriptstyle{\mspace{-3mu}}{(\scriptstyle{\mathrm{im}(\tau^\ast \mspace{-1.5mu}\: 
(C^N_{(t,j)})^{\mspace{-1mu}C_{p-1}} \to (C^N_{(t,j)})^{\mspace{-1mu}C_{p-1}}))}}}\\ 
& \cong H^c_0\mspace{-2mu}\bigl(\padics, \bigl(C^N_{(t,j)}\bigr)^{\mspace{-3mu}C_{p-1}}
\bigr) = \padics \widehat{\otimes}_{\padics[[\padics]]} 
\bigl(C^N_{(t,j)}\bigr)^{\mspace{-3mu}C_{p-1}}\\ 
& \cong \padics {\otimes}_{\padics[[\padics]]} 
\bigl(C^N_{(t,j)}\bigr)^{\mspace{-3mu}C_{p-1}},\end{align*} where the right-hand side 
of the second isomorphism is a continuous homology group (see 
\cite[Section 3.3]{Symonds}) and the last step is because 
$\bigl(C^N_{(t,j)}\bigr)^{\mspace{-3mu}C_{p-1}}$ is finite (and thus, a 
finitely generated object in the category of profinite $\padics[[\padics]]$-modules; 
see \cite[Proposition 5.5.3]{Ribes}). The isomorphism in the second step is not 
quite immediate, and it can be justified in a sleek way: since $\mathbb{Z}$ is 
an orientable discrete Poincar\'e duality group of dimension one and 
pro-$p$ good (in the sense of \cite[Section 3.1]{weigelabelian}; by \cite[Example 4.6, Lemma 4.7]{LinnellSchick}), the pro-$p$ completion $\padics$ is an orientable (profinite) Poincar\'e duality group at $p$ of dimension one, by \cite[Proposition 3.2]{weigelabelian} (here, for ``orientable (profinite) Poincar\'e duality group at $p$," we use the 
definitions in \cite[Section 4.4, page 394]{Symonds} and by \cite[Remark 2.2]{frattini}, 
these are equivalent to those used in \cite{weigelabelian}), and the desired 
isomorphism follows. 

Then the result follows by the manipulations that preceded this lemma. To 
understand the abstract $\padics[[\padics]]$-module structure of the $C_{p-1}$-fixed points of the coinduced module in the statement of the lemma (and of the various pieces involved in the manipulations), it is helpful to note that if $H$ is an arbitrary profinite group, then a $p$-torsion discrete $H$-module is canonically a discrete, and 
hence abstract, $\padics[[H]]$-module. 
\end{proof}

\section{A further reduction in the case when $G = \G$}\label{lastmaybe}
Let $V(0)$ be the mod $p$ Moore spectrum $M(p)$, and more generally, for each integer $i \geq 1$, let $M(p^i)$ be the mod $p^i$ Moore spectrum. By restriction of the $\G$-action, $C_{p-1}$ acts on $K(KU_p)$, so that there is the homotopy fixed 
point spectrum
\[K(KU_p)^{hC_{p-1}} = (K(KU_p))^{hC_{p-1}},\] and by 
\cite[page 1267]{rognesicm} (see \cite[pages 11-12]{jems} for a proof), 
the canonical map 
\[\holim_{i \geq 1} (K(L_p) \wedge M(p^i)) \xrightarrow{\,\simeq\,} 
\holim_{i \geq 1} (K(KU_p)^{hC_{p-1}} \wedge M(p^i))\] is a weak equivalence. It follows that the morphism
\[L_{V(0)}(K(L_p)) \xrightarrow{\,\simeq\,} L_{V(0)}(K(KU_p)^{hC_{p-1}})\] (between Bousfield localizations with respect to $V(0)$) is a weak equivalence, so that the natural map 
$K(L_p) \to K(KU_p)^{hC_{p-1}}$ is a $V(0)$-equivalence. The familiar cofiber sequence 
\[\Sigma^{2p-2}V(0) \xrightarrow{v_1} V(0) \xrightarrow{i_1} V(1)\] induces 
the commutative diagram
\[\xymatrix@C=.55cm{
K(L_p) \wedge \Sigma^{2p-2}V(0) \ar[r]  \ar[d] & K(L_p) \wedge V(0) \ar[r] \ar[d] & K(L_p) \wedge V(1) \ar[d]\\
K(KU_p)^{hC_{p-1}} \wedge \Sigma^{2p-2}V(0) \ar[r] & K(KU_p)^{hC_{p-1}} \wedge V(0) \ar[r]  & K(KU_p)^{hC_{p-1}} \wedge V(1)
}\] in which the rows are cofiber sequences. Since the leftmost and middle vertical maps are weak equivalences, the rightmost vertical map is a weak equivalence. 
Thus, for each $j \geq 0$, the map 
\[K(L_p) \wedge \Sigma^{-jd}V(1) \xrightarrow{\,\simeq\,} K(KU_p)^{hC_{p-1}} \wedge \Sigma^{-jd}V(1)\] 
is a weak equivalence. We apply this conclusion in the following way. 

There are the homotopy fixed point spectral sequences 
\[{^{\sharp \mspace{-2mu}}E_2^{s,t}} = H^s(C_{p-1}, \pi_t(K(KU_p) \wedge v_2^{-1}V(1))) \Longrightarrow 
\pi_{t-s}((K(KU_p) \wedge v_2^{-1}V(1))^{hC_{p-1}})\] and
\[{_{\scriptscriptstyle{j}\mspace{-7mu}}{^{\sharp \mspace{-3mu}}E_2^{s,t}}} \mspace{-2mu}=\mspace{-2mu} H^s\mspace{-2mu}(C_{p-1}, \pi_t(K(KU_p) \wedge \Sigma^{-jd}V(1))) \mspace{-2mu}\Rightarrow \mspace{-2mu}
\pi_{t-s}((K(KU_p) \wedge \Sigma^{-jd}V(1))^{hC_{p-1}}),\] for each $j \geq 0$. Since 
each $\pi_t(K(KU_p) \wedge \Sigma^{-jd}V(1))$ is $p$-torsion, both 
\[{^{\sharp \mspace{-2mu}}E_2^{s,t}} \cong 
\colim_{j \geq 0}{_{\scriptscriptstyle{j}\mspace{-7mu}}{^{\sharp \mspace{-3mu}}E_2^{s,t}}}\] and each $\displaystyle{{_{\scriptscriptstyle{j'}\mspace{-8mu}}{^{^{\scriptstyle{\sharp}} \mspace{-5mu}}E_2^{s,t}}}}$ vanish
for $s > 0$, $t \in \zee$, and $j' \geq 0$. As a consequence, 
\begin{align*}\pi_\ast((K(KU_p) \wedge v_2^{-1}V(1))^{hC_{p-1}}) & \cong 
(\pi_\ast(K(KU_p) \wedge v_2^{-1}V(1)))^{C_{p-1}},\\
(K(KU_p) \wedge v_2^{-1}V(1))^{hC_{p-1}} & \simeq \colim_{j \geq 0} 
(K(KU_p) \wedge \Sigma^{-jd}V(1))^{hC_{p-1}}, \ \text{and}\\ 
\pi_\ast((K(KU_p) \wedge \Sigma^{-jd}V(1))^{hC_{p-1}}) 
&\cong (\pi_\ast(K(KU_p) \wedge \Sigma^{-jd}V(1)))^{C_{p-1}}, \ j \geq 0
\end{align*} (the above equivalence of spectra (the middle line) is a special 
case of (\ref{colim}) from \cite[Theorem 1.7]{padicspectra2}, but here, [ibid.] is not needed and the conclusion follows from 
the vanishing properties stated above and \cite[Proposition 3.3]{mitchell}). Therefore (for the following deductions, we do not need the second isomorphism in 
$\mathrm{Ab}$ displayed above (which is indexed by $\{j \mid j \geq 0\}$); we state it here because of its intrinsic interest), we have the 
isomorphisms
\begin{align*}
(\pi_\ast(K(KU_p) \wedge v_2^{-1}V(1)))^{C_{p-1}} & \cong 
\colim_{j \geq 0} \pi_\ast((K(KU_p) \wedge \Sigma^{-jd}V(1))^{hC_{p-1}})\\
& \cong \colim_{j \geq 0} \pi_\ast(K(KU_p)^{hC_{p-1}}\wedge \Sigma^{-jd}V(1))\\
& \cong \pi_\ast(K(L_p) \wedge v_2^{-1}V(1)) \cong K(L_p)_\ast(V(1))[v_2^{-1}].\end{align*} Each of the spectra 
$K(L_p)$ and $K(KU_p)^{hC_{p-1}}$ have a natural action by $\padics$ and the 
map $K(L_p) \to K(KU_p)^{hC_{p-1}}$ is $\padics$-equivariant; thus, each of 
the above four isomorphisms is $\padics$-equivariant. 


\bibliographystyle{plain}

\end{document}